\documentclass[11pt]{article}
\usepackage{amsmath,amsthm,amssymb,hyperref}
\usepackage[margin=1.3in]{geometry}
\newtheorem{theorem}{Theorem}
\newtheorem{lemma}[theorem]{Lemma}
\newtheorem{proposition}[theorem]{Proposition}
\newtheorem{question}[theorem]{Question}

\hypersetup{
pdfauthor = {Boris Bukh},
pdftitle = {Radon partitions in convexity spaces},
pdfsubject = {Mathematics},
pdfkeywords = {Eckhoff's conjecture, convexity spaces, Radon partition, Radon number, Tverberg's theorem,
weak-epsilon nets},
pdfstartview = {FitH},
pdfpagemode = {None}}
\author{Boris Bukh\footnote{\texttt{B.Bukh@dpmms.cam.ac.uk}.
Centre for Mathematical Sciences,
Cambridge CB3 0WB, England
and Churchill College, Cambridge CB3 0DS, England.}}
\title{Radon partitions in convexity spaces%
\footnote{The paper is in public domain, and is not protected by copyright.}}
\date{}

\newcommand*{\R}{\mathbb{R}}                                   
\newcommand*{\F}{\mathcal{F}}                                  
\newcommand*{\T}{\mathcal{T}}                                  
\newcommand*{\bF}{\boldsymbol{\mathcal{F}}}                    
\newcommand*{\Nv}{\boldsymbol{\mathcal{N}}}                    
\newcommand*{\abs}[1]{\lvert #1\rvert}                         
\DeclareMathOperator{\conv}{conv}                              
\newcommand*{\veps}{\varepsilon}
\newcommand*{\eqdef}{\stackrel{\text{\tiny def}}{=}}

\begin{document}
\maketitle

\begin{abstract}
Tverberg's theorem asserts that every $(k-1)(d+1)+1$ points
in $\R^d$ can be partitioned into $k$ parts, so that the convex
hulls of the parts have a common intersection. Calder and Eckhoff 
asked whether there is a purely combinatorial deduction of Tverberg's 
theorem from the special case $k=2$. We dash the hopes of a purely 
combinatorial deduction, but show that the case $k=2$ does imply 
that every set of $O(k^2 \log^2 k)$ points admits a Tverberg partition into $k$ parts.
\end{abstract}

\section*{Introduction}
Radon's lemma \cite{radon_orig} states that every 
set $P$ of $d+2$ points in $\R^d$ can be partitioned
into two classes $P=P_1\cup P_2$ so that the convex 
hulls of $P_1$ and $P_2$ intersect.
Birch \cite{birch_plane} (for $d=2$) and Tverberg \cite{tverberg_orig} 
(for general $d$) extended Radon's theorem to
the analogous statement for partitions of a 
set into more than two parts: For a set $P\subset \R^d$
of $\abs{P}\geq (k-1)(d+1)+1$ points 
there is a partition
$P=P_1\cup\dotsb \cup P_k$ into $k$ parts, such that the 
intersection of the convex hulls
$\conv P_1\cap\dotsb\cap \conv P_k$ is non-empty. 
The bound of $(k-1)(d+1)+1$ is sharp, as
witnessed by any set of points in sufficiently general position.

Calder \cite{calder} conjectured and Eckhoff 
\cite{eckhoff_survey_orig} speculated that
Tverberg's theorem is a consequence of Radon's 
theorem in the context of abstract 
convexity spaces. The conjecture, which we now present, 
is commonly referred as 
``Eckhoff's conjecture'', and we will maintain 
this tradition to avoid additional confusion.
If true, the conjecture would have 
provided a purely combinatorial proof of Tverberg's theorem.
However, we will show that the conjecture is false.

A \emph{convexity space} on the ground set $X$ 
is a family $\F\subset 2^X$ of subsets
of $X$, called \emph{convex sets}, 
such as both $\emptyset$ and $X$ are convex, and intersection
of any collection of convex sets is convex. 
For example, the familiar convex sets in $\R^d$
form a convexity space on $\R^d$. Among the
other examples are axis-parallel boxes in $\R^d$, finite
subsets on any ground set, closed sets 
in any topological space (see the book \cite{vandevel} for
a through overview of convexity spaces). If 
the ground set $X$ in the convexity space $(X,\F)$ 
is clear from the context, we will speak simply 
of a convexity space $\F$. 
The \emph{convex hull} of a set $P\subset X$, 
denoted $\conv P$, is the intersection of all the convex sets
containing $P$. We write $\conv_{\F} P$ if the convexity
space is not clear from the context. The $k$-th \emph{Radon number} 
of $(X,\F)$ is the minimum 
natural number $r_k$, if it exists, so that every 
set $P\subset X$ of at least $r_k$ points admits a 
partition $P=P_1\cup\dotsb\cup P_k$ into $k$ 
parts whose convex hulls have an element in common.
It is not hard to show\footnote{According to \cite{eckhoff_survey} 
it was first shown by R.E.Jamison (1976).
The first published proofs appear to be 
in \cite{doignon_reay_sierksma} and \cite{jamison_r3}.}
that if $r_2$ is finite, then so is $r_k$.  Eckhoff's 
conjecture states that $r_k\leq (k-1)(r_2-1)+1$ in every 
convexity space. The conjecture has been proved for $r_2=3$ 
by Jamison \cite{jamison_r3}, and for convexity space 
with at most $2r_2$ points by Sierksma and Boland \cite{sierksma_boland}. 
In section \ref{sec_r23} we reproduce a version of Jamison's proof.

The best bounds on $r_k$ are
\begin{align*}
r_{k_1k_2}&\leq r_{k_1}r_{k_2}\qquad&\text{(due to Jamison \cite{jamison_r3})},\\
r_{2k+1}&\leq (r_2-1)(r_{k+1}-1)+r_k+1\qquad&\text{(due to Eckhoff \cite{eckhoff_survey})}.
\end{align*}
In particular,
\begin{equation}\label{prevbest}
r_k\leq k^{\lceil\log_2 r_2\rceil}.
\end{equation}

The following result improves on \eqref{prevbest}.
\begin{theorem}\label{thmklogk}
Let $(X,\F)$ be a convexity space, and assume that 
$r_2$ is finite. Then
\begin{equation*}
r_k\leq c(r_2) k^2\log^2 k,
\end{equation*}
where $c(r_2)$ is a constant that depends only on $r_2$.
\end{theorem}
Though this bound is not far from Eckhoff's 
conjecture, the conjecture itself is false.
\begin{theorem}\label{thmcounterexample}
For each $k\geq 3$ there is a convexity space $(X,\F)$ such 
that $r_2=4$, but $r_k\geq 3(k-1)+2$.
\end{theorem}

Despite the failure of Eckhoff's conjecture, we have been unable to 
rule out that the convexity spaces with finite $r_2$ 
might behave similarly to Euclidean spaces. It is conceivable 
that $r_k$ is bounded by a linear function of $k$ for each $r_2$.
Moreover, it is possible that other results from combinatorial 
convexity extend to such spaces.
For instance, Radon proved the lemma now bearing his 
name to give an alternative proof of Helly's theorem that if in 
some family of convex sets in $\R^d$ every $d+1$ sets intersect, 
then all of them do.
One of the easy but startling consequences of 
Helly's theorem is the centrepoint theorem. The centrepoint 
theorem asserts that
for every finite set $P\subset \R^d$ there is a point $p\in \R^d$ (the ``centrepoint'') such that every convex set containing more 
than $\frac{d}{d+1}\abs{P}$ points of $P$ also contains $p$.
Both the deduction of Helly's theorem from Radon's 
theorem, and the deduction of centrepoint theorem from
Helly's theorem remain valid in the context of the 
convexity spaces with finite $r_2$.
This prompts the following question:
\begin{question}[Weak epsilon-nets]\label{weak_question}
Suppose $(X,\F)$ is a convexity space 
with finite $r_2$. Let $\veps>0$ be given.
Let $P\subset X$ be a set of points in the space. 
Is there a set $N$ of $\abs{N}\leq f(\epsilon,r_2)$ points 
such that every convex set $S$ containing more 
than $\veps \abs{P}$ points of $P$ also contains
at least one point of $N$?
\end{question}
The set $N$ as in the question above is 
called a \emph{weak $\veps$-net} (with respect to convex sets) for $P$. 
In $\R^d$ it is known that there are weak $\veps$-nets of 
size only $(1/\veps)^d \log^{c_d} (1/\veps)$.
The discussion above shows that the answer to 
the question is positive if $\epsilon>1-1/(r_2-1)$.
It is unclear whether the weak $\veps$-nets 
of size $f(\veps,r_2)$ exist for any $\epsilon<1-1/(r_2-1)$.

B\'ar\'any \cite{barany} showed that if $P$ is an $n$-point set in $\R^d$, then 
there is a point $p$ in $c_d \binom{n}{d+1}$ 
of all the $\binom{n}{d+1}$ simplices spanned by $P$, 
where $c_d$ is a positive constant that depends 
only on $d$. In $\R^1$, it is immediate 
that $c_1=1/2$ is admissible, and is best 
possible. The situation for convexity spaces 
with bounded $r_2$ is again unclear, except if $r_2=3$:
\begin{proposition}[Selection theorem]\label{prop_r23sel}
Let $(X,\F)$ be a space with $r_2=3$. 
Let $P\subset X$ be point set.
Then there is a point $p\in X$ that is 
contained in at least $\frac{1}{3}\binom{n}{2}+O(n)$ of
all the sets $\conv \{x,y\}$.
\end{proposition}
\begin{question}
Does the preceding proposition hold with $1/2$ in place of $1/3$?
\end{question}
The standard greedy argument of Alon, B\'ar\'any, F\"uredi, Kleitman
\cite[Section~8]{abfk} shows that the selection theorem implies an affirmative
answer to Question~\ref{weak_question}. In particular, it gives
$f(\epsilon,3)\leq O\bigl((1/\epsilon)^2\bigr)$, which is probably not sharp.

The rest of the paper is organized as follows. 
In section \ref{sec_nerves} we introduce our only 
technical tool, the nerves of convex sets. 
In lemma~\ref{nervetospace} we will show that the 
nerves encode all the information 
about the convexity space that we 
need. In section \ref{sec_counterexample} we present a counterexample 
to Eckhoff's conjecture. It is then followed in section \ref{sec_klogk}
by the proof of Theorem~\ref{thmklogk}. 
We conclude the paper with a short 
discussion of convexity spaces with $r_2=3$.

\section{Nerves}\label{sec_nerves}
Let $P$ be a set of points in a some convexity 
space. We associate to $P$ a collection $\Nv(P)$
of subsets of $2^P$. A family 
$\F\subset 2^P$ belongs to $\Nv(P)$ if and 
only if the intersection
$\bigcap_{S\in \F} \conv S$ is 
non-empty. In the conventional terminology one would say that 
the collection $\Nv(P)$ is the 
nerve of the family of convex sets $\{\conv S : S\subset P\}$. 
Since we will not use the 
nerves of any other families of sets, in this paper 
we abuse the language and say that $\Nv(P)$ is the nerve of $P$.

\begin{proposition}\label{nervechar}
If $\Nv=\Nv(P)$, then $\Nv$ satisfies the following properties:
\renewcommand{\theenumi}{(N\arabic{enumi})}
\renewcommand{\labelenumi}{\theenumi}
\begin{enumerate}
\item \label{Nv_downset} $\Nv$ is a downset: if $\F\in \Nv$ and 
$\F'\subset \F$, then
$\F'\in \Nv$.
\item \label{Nv_complete} If $\F$ is in $\Nv$, then so is 
$\hat{\F}\eqdef\{S' : S'\supset S\in  \F\}$.
\item \label{Nv_point} For every $p\in P$ the family $\F(p)\eqdef\{ S : p\in S \}$ is in $\Nv$.
\item \label{Nv_partition} The set $P$ can be partitioned into $k$ parts $P=P_1\cup\dotsb\cup P_k$ 
so that $(\conv P_1)\cap\dotsb\cap(\conv P_k)\neq \emptyset$ if and only
if there is a family $\{P_1,\dotsc,P_k\}\in \Nv$ consisting of $k$
disjoint sets.
\item \label{Nv_rt} If $r_t$ exists, then for every set of $r_t$ families 
$\bF=\{\F_1,\dotsc,\F_{r_t}\}\subset \Nv$
there is a partition $\bF=\bF_1\cup\dotsb \cup\bF_t$ of $\bF$ into $t$ parts so
that $(\bigcap \bF_1)\cup \dotsb\cup (\bigcap \bF_t)\in \Nv$.
\end{enumerate}
\end{proposition}
\begin{proof}
The first properties four properties are immediate from the definition
of $\Nv(P)$.

The final property is easy too: Suppose
$\bF=\{\F_1,\dotsc,\F_{r_t}\}$ is given. Let $q_i$ be any point in
$\bigcap_{S\in\F_i} \conv S$. The set $Q=\{q_1,\dotsc,q_{r_t}\}$
of $r_t$ points can be partition into $t$ 
parts $Q=Q_1\cup\dotsb\cup Q_t$ so that
$(\conv Q_1)\cap\dotsb\cap(\conv Q_t)$ is non-empty, thus containing
some point $p$. The partition $Q=Q_1\cup\dotsb\cup Q_t$ naturally
induces the partition $\bF=\bF_1\cup\dotsb\cup \bF_t$. It is easy
to see that the point $q$ belongs to $\bigcap_{S\in \cap \bF_i}\conv S$
for each $i=1,\dotsc,t$. 
\end{proof}

Thanks to the following lemma we can avoid the 
convexity spaces in the rest of the paper, and work exclusively with
nerves.
\begin{lemma}\label{nervetospace}
Let $P$ be a set, and let $\Nv$ be a collection of subsets of $2^P$ that satisfies
the first three properties in the Proposition~\ref{nervechar}. Then there are
a ground set $X\supset P$ and a convexity space on $X$ so that $\Nv(P)=\Nv$.
\end{lemma}
\begin{proof}
For an arbitrary family $\F$ let $C(\F)=\{\F'\in \Nv : \F\subset \F'\}$,
and denote by $\mathcal{C}$ the family of all the sets of the form $C(\F)$.
Put $X=\Nv$. We claim that $\mathcal{C}$ forms a desired convexity space on $X$.
It is clear that $\emptyset, X\in \mathcal{C}$.
Since $C(\F_1)\cap C(\F_2)=C(\F_1\cup \F_2)$, and similarly for intersections
of more than two sets, the collection $\mathcal{C}$ indeed forms a convexity space on $X$.
Define $\phi\colon P\to X$ by $\phi(p)=\F(p)$. The map $\phi$ is well-defined
by property \ref{Nv_point}, and provides the
embedding of $P$ into $X$. We need to check that $\Nv(\phi(P))=\phi(\Nv)$

Since $\F(p)\in C(\F)$ if and only if $\F\subset \F(p)$, it follows that
$\{\F(p_1),\dots,\F(p_t)\}\subset C(\F)$ precisely when $\F\subset \bigcap \F(p_i)$.
Hence, if $P'\subset P$, then 
\begin{align*}
\conv_{\mathcal{C}}(\phi(P'))=\bigcap_{\phi(P')\subset C(\F)} C(\F)=
\bigcap_{\F\subset \bigcap_{p\in P'} \F(p)} C(\F)
=C\left(\bigcap_{p\in P'} \F(p)\right).
\end{align*}
Hence, $\F\in \conv_{\mathcal{C}}(\phi(P'))$ if and only if $\{P : P'\subset P\} \subset \F$.
The intersection $\bigcap_{P\in \F'} \conv_{\mathcal{C}}(\phi(P))$
is non-empty if and only if there is an $\F\in \Nv$ so that $\hat{\F}'\subset \F$.
Thus by the properties \ref{Nv_downset} and \ref{Nv_complete} 
\begin{equation*}
\bigcap_{P\in \F'} \conv_{\mathcal{C}}(\phi(P))\neq \emptyset\iff F'\in \Nv.
\end{equation*}
Therefore $\Nv_{\mathcal{C}}(\phi(P))=\phi(\Nv)$ as claimed.
\end{proof}

\section{Counterexample to Eckhoff's conjecture}\label{sec_counterexample}
\begin{proof}[Proof of Theorem~\ref{thmcounterexample}]
We shall use the Lemma~\ref{nervetospace} to construct the requisite convexity space.
Let $P=[3(k-1)+1]$. Consider the three kinds of families:
\begin{align*}
A[x]&=\bigl\{ \{x\} \bigr\}\cup \binom{P}{4},\\
B[xy:zw]&=\bigl\{ \{x,y\}, \{z,w\} \bigr\}\cup\bigl\{S\!\in\! \binom{P}{3} : \{x,y,z,w\}\cap S\neq \emptyset \bigr\} \cup \binom{P}{4},\text{ distinct }x,y,z,w\\
C[xy]&=\bigl\{ \{x,y\} \bigr\}\cup \binom{P}{3},\qquad x,y\text{ are distinct}.
\end{align*}
Here $x,y,z,w$ are elements of $P=[3(k-1)+1]$.
Let $\hat{A[x]},\hat{B}[xy:zw]$ and $\hat{C}[xy]$ be as in Proposition~\ref{nervechar}
property \ref{Nv_complete}. 
Let $\Nv$ consist of all the families, $\hat{A[x]},\hat{B}[xy:zw]$ and $\hat{C}[xy]$
and all their subfamilies. Let $\Nv_0$
consist only of families $\hat{A}[x]$, $\hat{B}[xy:zw]$ and $\hat{C}[xy]$.
 As $\Nv$ automatically satisfies properties \ref{Nv_downset} and \ref{Nv_complete} 
in Proposition~\ref{nervechar}
and $\F(p)\subset \hat{A}[p]$, by Lemma~\ref{nervetospace} it is a nerve of some convexity space. 
As $k\geq 3$, no family of the form $A[x],B[xy:zw]$ or $C[xy]$ contains $t$ disjoint sets.
From that it follows that none of $\hat{A}[x],\hat{B}[xy:zw]$ or $\hat{C}[xy]$
contain $k$ disjoint sets either, and same holds for every family in $\Nv$.
Therefore, to establish the theorem it remains to verify the property \ref{Nv_rt} with 
$r_2=4$.

As $\hat{A}$-, $\hat{B}$- and $\hat{C}$-families are the maximal families in $\Nv$, 
it suffices to show that whenever $\bF=\{\F_1,\dotsc,\F_4\}$ is a collections
of four families in $\Nv_0$, then there is a partition $\bF=\bF_1\cup\bF_2$ so that
$(\bigcap \bF_1)\cup(\bigcap \bF_2)$ is contained in some $\F\in \Nv_0$. 

To every family $\F$ we associate a subset 
\begin{equation*}
e(\F)=\F\cap \binom{P}{2}.
\end{equation*}
That is 
\begin{align*}
e(\hat{A}[x])&=\bigl\{ \{x,y\} : y\in P\setminus\{x\}\bigr\}\\
e(\hat{B}[xy:zw])&=\bigl\{\{x,y\},\{z,w\}\bigr\},\\
e(\hat{C}[xy])&=\bigl\{ \{x,y\}\bigr\}.
\end{align*}
Note that $e(\F_1\cap \F_2)=e(\F_1)\cap e(\F_2)$.
It is convenient think of $e(\F)$ as an edge of a hypergraph 
on the ground set $\binom{P}{2}$.

Note that if $\F_1,\F_2\in\Nv_0$
are two distinct families, then $\F_1\cap \F_2$ is contained in a $\hat{C}$-set. 
Moreover, if $e(\F_1)\cap e(\F_2)=\emptyset$, then $\F_1\cap\F_2$ in contained
in $\binom{P}{3}$.

Suppose $\F_1,\dotsc,\F_4$ are four families in $\Nv_0$. If $e(\F_1)\cap e(\F_2)=\emptyset$,
then $\F_1\cap \F_2\subset \binom{P}{3}$ and $\F_3\cap \F_4\subset \hat{C}[xy]$ for some $x,y$. Hence 
$(\F_1\cap \F_2)\cup(\F_3\cap \F_4)\subset \binom{P}{3}\cup \hat{C}[xy]=\hat{C}[xy]$. We may thus assume
that $e(F_1)\cap e(F_2)$ is non-empty, and similarly for other pairs of sets $\F_1,\dotsc,\F_4$.

There are five cases according to the number of $\hat{A}$-families among the four
families $\F_1,\dotsc\F_4$.

\textbf{There no $\hat{A}$-families:} Since every two families meet, and $e(\F_1),\dotsc,e(\F_4)$
contain $1$ or $2$ vertices each, it follows that $e(\F_1),\dotsc,e(\F_4)$ must have a common vertex,
say $\{1,2\}\in \binom{P}{2}$. Then $(\F_1\cap \F_2)\cup(\F_3\cap \F_4)\subset \hat{C}[12]$.

\textbf{There is a single $\hat{A}$-family $\F_1$:} As $e(\F_2)$, $e(\F_3)$ and $e(\F_4)$ 
pairwise meet, they either have a vertex in common, or $\F_2,\F_3,\F_4$ are $\hat{B}$-families,
and $e(\F_2),e(\F_3),e(\F_4)$ form a triangle. However, they cannot form the triangle
because $e(\F_1)$ would not meet each of $e(\F_2)$, $e(\F_3)$ and $e(\F_4)$. Thus,
$e(\F_1)\cap\dotsb\cap e(\F_4)$ is non-empty, and equals to say $\{1,2\}\in\binom{P}{2}$. Then
$(\F_1\cap \F_2)\cup(\F_3\cap \F_4)\subset \hat{C}[12]$.

\textbf{There are two $\hat{A}$-families $\F_1$ and $\F_2$:} 
The intersection $e(\F_3)\cap e(\F_4)$ 
contains just one element, say $\{x,y\}$. If $\F_1$ and $\F_2$ are 
just $\hat{A}[x]$ and $\hat{A}[y]$, then 
$(\F_1\cap \F_2)\cup (\F_3\cap \F_4)\subset \hat{C}[xy]$. 
If $\F_1=\hat{A}[z]$ and $z\not\in\{x,y\}$,
then it necessarily follows that $e(\F_3)=\bigl\{\{x,y\},\{z,w_3\}\bigr\}$ 
and $e(\F_4)=\bigl\{\{x,y\},\{z,w_4\}\bigr\}$ for some $w_3$ and $w_4$.
Thus $\F_2$ is either $\hat{A}[x]$ or $\hat{A}[y]$.
In either case $(\F_1\cap F_3)\cup(\F_2\cap \F_4)\subset \hat{B}[xy:zw_3]$.

\textbf{There are three $\hat{A}$-families $\F_1$, $\F_2$ and $\F_3$:} 
As $e(\F_4)$ has to meet all of $e(\F_1)$, $e(F_2)$, $e(F_3)$, it must
be that $\F_4$ is a $\hat{B}$-family, implying that $(\F_1\cap \F_2\cap\F_3)\cup \F_4=\F_4$.

\textbf{All four families are $\hat{A}$-families:} Say,
$\F_1=\hat{A}[x]$, $\F_2=\hat{A}[y]$, $\F_3=\hat{A}[z]$ and $\F_4=\hat{A}[w]$.
In that case $(\F_1\cap \F_2)\cup(F_3\cap \F_4)\subset \hat{B}[xy:zw]$.
\end{proof}

\section{Upper bound on Radon numbers}\label{sec_klogk}
The main ingredient in the proof of theorem~\ref{thmklogk} 
is a version of Kruskal--Katona theorem from \cite{bukh_kk}.
A \emph{$d$-dimensional $r$-uniform 
family} is a $d$-tuple of $r$-element sets. In other words, if we 
denote by $\binom{X}{r}$ the family of all $r$-element subsets of $X$, 
then $d$-dimensional $r$-uniform family is a subset of $\binom{X}{r}^d$.
A \emph{shadow} of such a family $\F\subset \binom{X}{r}^d$ is defined to be
\begin{equation*}
\partial \F \eqdef \bigl\{ (S_1\setminus\{x_i\},\dotsc,S_d\setminus\{x_d\} ) :  
   (S_1,\dotsc,S_d) \in \F, \text{and }x_i\in S_i\text{ for }i=1,\dotsc,d\bigr\}.
\end{equation*}
\begin{lemma}[Theorem~1 of \cite{bukh_kk}]\label{kk_lemma}
Suppose $\F\subset \binom{X}{r}^d$ is a $d$-dimensional $r$-uniform family
of size
\begin{equation*}
\abs{\F}=\binom{x}{r}^d,
\end{equation*}
where $x\geq r$ is a real number. Then
\begin{equation*}
\abs{\partial \F}\geq \binom{x}{r-1}^d.
\end{equation*}
\end{lemma}

In addition to multidimensional Kruskal--Katona theorem, we shall need four lemmas. 
The first two lemmas are a bound on Tur\'an numbers of hypergraphs
and a bound on the independence numbers of graphs in which every subgraph
have a large independence number.
\begin{lemma}[\cite{decaen_turan}]\label{turan_bound}
If $H$ is an $s$-uniform hypergraph on $n$ vertices with fewer than
$\binom{l-1}{s-1}^{-1}\frac{n-l+1}{n-s+1}\binom{\abs{H}}{s}$ edges, then
$H$ contains an independent set on $l$ vertices.
\end{lemma}
\begin{lemma}[Special case of Theorem 2.1 from \cite{alon_sudakov}]\label{lemma_alon_sudakov}
Let $t<s\leq 2s-3$, and let $G$ be a graph on $n$ vertices. Suppose that every set of 
size $s$ contains an independent set of size $t$. Then $G$ contains
an independent set of size $n-s+1$.
\end{lemma}

Our third lemma is purely computational.
We say that a tuple $(S_1,\dotsc,S_d)\in\binom{P}{a}^d$ is \emph{$r$-good}
if there are $r$ pairwise disjoint sets $S_{i_1},\dotsc,S_{i_r}$ among $S$'s.
\begin{lemma}\label{lemma_rgood}
Let $P$ be a finite set. There are at most 
\begin{equation*}
C(d)(a^2/\abs{P})^{d-r+1} \binom{\abs{P}}{a}^d
\end{equation*}
$r$-bad tuples in $\binom{P}{a}^d$, where $C(d)$ is a constant
that depends only on $d$.
\end{lemma}
\begin{proof}
For $S=(S_1,\dots,S_d)\in\binom{P}{a}^d$ let $G[S]$ be a graph on $\{1,\dotsc,d\}$ with
$ij$ forming an edge if $S_i\cap S_j\neq\emptyset$.
A tuple $S$ is $r$-bad if and only if the independence number
of $G[S]$ is less than $r$. Suppose that the largest forest in 
$G[S]$ has $m$ edges, then by contracting these edges we obtain
an independent set of size $d-m$. Thus if a tuple 
$S$ is $r$-bad, then $G[S]$ contains a forest $F$ with $d-r+1$ edges. 
We say that the forest $F$ \emph{witnesses} that $S$ is $r$-bad.

Fix a forest $F$. We shall bound the number of $r$-bad tuples $S$ for 
which $F$ is a witness that $S$ is $r$-bad. 
Let $v_1,\dotsc,v_d$ be a relabelling of $\{1,\dotsc,d\}$ so that
in $F$ the vertex $v_i$ is adjacent to at most one vertex $v_j$ with $j<i$.
Pick $S_1,\dotsc,S_d$ uniformly at random from $\binom{P}{a}$.
If $v_i$ is adjacent to some $v_j$ with $j<i$ let $E_i$ be the 
event that $S_i\cap S_j\neq \emptyset$. If $v_j$ is adjacent to
none $v_j$ with $j<i$ let $E_i$ be the event that holds with probability
$1$. Then
\begin{align*}
\Pr[F\text{ is a witness that }S\text{ is $r$-bad}]=
\prod_{i=1}^d \Pr[E_i|E_1,\dotsc,E_{i-1}]=
\prod_{i=1}^d \Pr[E_i]\leq (a^2/\abs{P})^{d-r+1}.
\end{align*}
As the number of forests on $d$ vertices depends only on $d$, the lemma follows
by the union bound.
\end{proof}

Finally, the third lemma that we need is a restatement of Jamison's upper 
bound $r_{2^t}\leq r_{2}^t$ in terms of nerves. We include the 
proof for completeness.
\begin{lemma}
Suppose $P$ a set in a convexity space, and $\Nv=\Nv(P)$ is its nerve. 
Then for every set $P'\subset P$ of size
$\abs{P'}=r_2^t$ there is a family $\F\in \Nv$ containing $2^t$ disjoint
subsets of $P'$.
\end{lemma}
\begin{proof}
The proof is by induction on $t$. The base case $t=0$ is trivial. Suppose $t\geq 1$.
Let $P'=P_1'\cup\dotsb\cup P_{r_2}'$ be a partition of $P'$ into sets of size
$r_2^{t-1}$. By the induction hypothesis, there are families 
$\F_1,\dotsc,\F_{r_2}$ such that each $\F_i$ contains $2^{t-1}$ disjoint subsets of 
$P_i'$. Let these subsets be $R_{i,1},\dotsc,R_{i,2^{t-1}}$. 
By property \ref{Nv_rt} of the proposition~\ref{nervechar}, there is a 
a set $I\subset [r_k]$ so that
\begin{equation*}
\bigcap_{i\in I} \F_i \cup \bigcap_{i\not\in I} \F_i\in \Nv.
\end{equation*}
By property~\ref{Nv_complete} the intersection 
$R_j=\bigcap_{i\in I} R_{i,j}$ is in $\bigcap_{i\in I} \hat{\F}_i$
for each $j=1,\dotsc,2^{t-1}$. The sets $R_j$ are $2^{t-1}$ disjoint subsets
of $\bigcup_{i\in I} P'_i$. Similarly one obtains $2^{t-1}$ disjoint
subsets of $\bigcup_{i\not\in I} P'_i$, for the total of $2^t$ disjoint
subsets of $P'$.
\end{proof}
\begin{proof}[Proof of theorem~\ref{thmklogk}]
It suffices to show that for every nerve $\Nv$ on $\abs{P}=k^2 \log^2 k$ points
there are $k$ disjoint sets $S_1,\dotsc,S_k\subset P$ and a family $\F$ 
that contains all of these sets.

For brevity we shall write $r=r_2$ and $t=1+\lceil \log_2 r\rceil$. 
Define a $(2r-3)$-dimensional family $\T\subset \binom{P}{r^t}^{2r-3}$ as follows: 
A tuple $(S_1,\dotsc,S_{2r-3})\in  \binom{P}{r^t}^{2r-3}$ is in $\T$ if there is a 
family $\F\in \Nv$
such that $\{S_1,\dotsc,S_{2r-3}\}\subset\F$.
Let $P_0\subset P$ be any $(2r-3)r^t$-element
subset of $P$. Let $P'\subset P_0$ be an arbitrary $r^t$-element
subset of $P_0$. By the preceding lemma there is a family $\F$
that contains $2^t$ disjoint subsets of $P'$. Since $2r-3\leq 2^t$, 
by property \ref{Nv_complete} it follows that $\F$ contains $2r-3$ 
disjoint subsets of size $r^t$ each that 
partition $P_0$. In other words, $P_0$ gives rise to at least
one tuple in $\T$. Since $P_0$ is an arbitrary $(2r-3)r^t$-element
subset of $P$, we conclude that 
\begin{equation*}
\abs{\T}\geq \binom{\abs{P}}{(2r-3)r^t}\geq c_1(r)\abs{P}^{(2r-3)r^t} \geq 
\binom{\abs{P}}{r^t}^{2r-3}-\binom{(1-c_2(r)) \abs{P}}{r^t}^{2r-3}
\end{equation*}
for some positive constants $c_1(r),c_2(r)$ that depend only on $r$.

Let $m=\lceil \log k/c_2(r)\rceil$. Define a $(2r-3)$-dimensional family 
$\T'\subset \binom{P}{mr^t}^{2r-3}$ in the same way as $\T$ was
defined: namely, $S\in \T'$ if there is an $\F\in \Nv$ such 
that $S\subset \F$. Note that the property \ref{Nv_complete} implies that
if $S\in \binom{P}{mr^t}^{2r-3}$ is not in $\T'$, then neither is any
family obtained from $S$ by removing some elements from each set in $S$.
Lemma~\ref{kk_lemma} applied to the complement of $\T'$ yields
\begin{equation*}
\abs{\T'}\geq \binom{\abs{P}}{mr^{t+1}}^{2r-3}-\binom{(1-c_2(r)) \abs{P}}{mr^t}^{2r-3}.
\end{equation*}

Let $H\subset \binom{\binom{P}{mr^t}}{2r-3}$ be a $(2r-3)$-uniform hypergraph
on $\binom{P}{mr^t}$ with edges
\begin{equation*}
\{S_1,\dotsc,S_{2r-3}\}\in H \iff (S_1,\dotsc,S_{2r-3})\in \T'\text{ and }(S_1,\dotsc,S_{2r-3})\text{ is $r$-good}.
\end{equation*}
By Lemma~\ref{lemma_rgood}, it follows that
\begin{align*}
\abs{H}&\geq \frac{1}{(2r-3)!}\left(\abs{\T'}-c_3(r)(m^2r^{2t}/\abs{P})^{r-2} \binom{\abs{P}}{mr^t}^{2r-3}\right)\\
&\geq \binom{\binom{\abs{P}}{mr^t}}{2r-3}\left(1-(1-c_2(r))^{(2r-3)mr^t}-c_4(r)(m^2/\abs{P})^{r-2}\right)
\end{align*}
Since $m>\log k/c_2(r)$, and $\abs{P}\geq (9c_4(r))^{1/(r-2)}m^2k^2$ it follows that the density of $H$ is
\begin{equation*}
\abs{H}/\binom{\binom{\abs{P}}{mr^t}}{2r-3}\geq 1-k^{(2r-3)r^t}-(3k)^{-(2r-4)} \geq 1-(2k)^{-(2r-4)}
\end{equation*}
for $k$ large enough.

By Lemma~\ref{turan_bound} the hypergraph $H$ contains a clique on $2k$ vertices. Let 
$S_1,\dotsc,S_{2k}\in \binom{P}{mr^t}$ be the vertices of this clique.
Since edges of $H$ are $r$-good among every $2r-3$ of these $2k$ sets 
there are $r$ that are pairwise disjoint. Thus, by Lemma~\ref{lemma_alon_sudakov}
there are $k$ of them, say $S_1,\dotsc,S_k$, that are pairwise disjoint.

We claim that for every $I\subset [k]$ there is
a family $\F_I\in \Nv$ that contains
$S_i$ for every $i \in I$. The proof is by induction
on $\abs{I}$ starting with $\abs{I}=2r-3$.
If $\abs{I}=2r-3$, then the claim holds because
$\{S_i : i\in I\}$ is an edge in $H$.
Suppose $\abs{I}>2r-3$. Pick any $r$ distinct
$\abs{I}-1$-element subsets $I_1,\dotsc,I_r$ of $I$.
Then by by property \ref{Nv_rt} applied to families
$\F_{I_1},\dotsc,\F_{I_r}$ it follows that
there is a $J\subset [r]$ so that
$\F=(\bigcap_{j\in J} \F_{I_j})\cup(\bigcap_{j\not\in J} \F_{I_j})\in \Nv$.
Since the family $\F$ contain $\F_i$ for every $i\in I$, we may put $\F_I=\F$.

Finally, the family $\F_{[k]}$ contains $k$ disjoint sets $S_1,\dotsc,S_k$,
as required.
\end{proof}

\section{Convexity spaces with \texorpdfstring{$r_2=3$}{r2=3}}\label{sec_r23}
The space with $r_2=3$ are especially nice because of the following
lemma, which is implicit in \cite{jamison_r3}.

\begin{lemma}
Let $P$ be a set in a convexity space with $r_2=3$, 
and let $\Nv=\Nv(P)$ be its nerve. Then there is
a family $\F_p\in\Nv$
for each $p\in P$, and these families satisfy
\begin{enumerate}
\renewcommand{\theenumi}{(J\arabic{enumi})}
\renewcommand{\labelenumi}{\theenumi}
\item \label{r3_point} $\{p\}\in \F_p$.
\item \label{r3_excl} If $p,q,r$ are any three points of $P$, then
either $\{p,q\}\in \F_r$ or $\{p,r\}\in F_q$ or
$\{q,r\}\in \F_p$.
\item \label{r3_transf} If $\{q,r\}\in \F_p$ and $\{r,s\}\in \F_q$, then $\{r,s\}\in\F_p$.
\end{enumerate}
\end{lemma}
\begin{proof}
Let $\F_p$ be a maximal family containing $\{p\}$. Then
the other conditions follow from the property \ref{Nv_rt} applied
to the triple of families $\F_p,\F_q,\F_r$.
\end{proof}
\begin{proof}[Proof of Proposition~\ref{prop_r23sel}]
Let $I=\{(p,q,r) : p\in \conv\{q,r\}\}$.
Since there are $\binom{n}{3}$ triples $\{p,q,r\}$,
each of which contributes at least at least one
element $I$, the proposition follows by the pigeonhole
principle.
\end{proof}
Since Jamison's proof of Eckhoff's conjecture is especially short
in the language of nerves, we include it:
\begin{theorem}
If $r_2=3$, then $r_k\leq 2(k-1)+1$.
\end{theorem}
\begin{proof}
Suppose $\abs{P}=2(k-1)+1$. We shall show that
one of $\F_p$ contains $k$ pairwise disjoint sets.
We claim that there is a pair of elements $p,q\in P$
so that $\{p,q\}\in \F_r$ for every $r\neq p,q$.
Indeed, it is true if $\abs{P}\leq 3$. If $\abs{P}\geq 4$,
and $s$ is any element of $\F_p$, then by induction there 
is a $p,q\in P\setminus\{s\}$ so that $\{p,q\}\in \F_r$ for
every $r\neq p,q,s$. If in addition $\{p,q\}\in \F_s$, then
we are done. Otherwise by property \ref{r3_excl} either 
$\{p,s\}\in\F_q$ or $\{q,s\}\in \F_p$. Say $\{p,s\}\in\F_q$.
Then by property \ref{r3_transf} applied to $\{p,q\}\in \F_r$
and either $\{p,s\}\in \F_q$ we conclude that $\{p,s\}$ is
in every $\F_r$, $r\neq p,s$. The claim is proved.

Let $p,q$ be a pair of element so that $\{p,q\}\in \F_r$ for $r\neq p,q$. 
By the induction hypothesis applied to $P\setminus\{p,q\}$ there is $r\in\{p,q\}$
so that $\F_r$ contains $k-1$ disjoint sets that are also disjoint
from $\{p,q\}$. Together with $\{p,q\}$ these form
a desired family of disjoint sets.
\end{proof}

\bibliographystyle{alpha}
\bibliography{eckhoff}

\begin{thebibliography}{ABFK92}

\bibitem[ABFK92]{abfk}
Noga Alon, Imre B{\'a}r{\'a}ny, Zolt{\'a}n F{\"u}redi, and Daniel~J. Kleitman.
\newblock Point selections and weak {$\epsilon$}-nets for convex hulls.
\newblock {\em Combin. Probab. Comput.}, 1(3):189--200, 1992.
\newblock \url{http://www.tau.ac.il/~nogaa/PDFS/abfk3.pdf}.

\bibitem[AS07]{alon_sudakov}
Noga Alon and Benny Sudakov.
\newblock On graphs with subgraphs having large independence numbers.
\newblock {\em J. Graph Theory}, 56(2):149--157, 2007.
\newblock \url{http://www.math.ucla.edu/~bsudakov/erdos-hajnal.pdf}.

\bibitem[B{\'a}r82]{barany}
Imre B{\'a}r{\'a}ny.
\newblock A generalization of {C}arath\'eodory's theorem.
\newblock {\em Discrete Math.}, 40(2-3):141--152, 1982.

\bibitem[Bir59]{birch_plane}
B.~J. Birch.
\newblock On {$3N$} points in a plane.
\newblock {\em Proc. Cambridge Philos. Soc.}, 55:289--293, 1959.

\bibitem[Buk10]{bukh_kk}
Boris Bukh.
\newblock Multidimensional {K}ruskal--{K}atona theorem.
\newblock arXiv, Sep 2010.

\bibitem[Cal71]{calder}
J.~R. Calder.
\newblock Some elementary properties of interval convexities.
\newblock {\em J. London Math. Soc. (2)}, 3:422--428, 1971.

\bibitem[dC83]{decaen_turan}
D.~de~Caen.
\newblock Extension of a theorem of {M}oon and {M}oser on complete subgraphs.
\newblock {\em Ars Combin.}, 16:5--10, 1983.

\bibitem[DRS81]{doignon_reay_sierksma}
Jean-Paul Doignon, John~R. Reay, and Gerard Sierksma.
\newblock A {T}verberg-type generalization of the {H}elly number of a convexity
  space.
\newblock {\em J. Geom.}, 16(2):117--125, 1981.

\bibitem[Eck79]{eckhoff_survey_orig}
J{\"u}rgen Eckhoff.
\newblock Radon's theorem revisited.
\newblock In {\em Contributions to geometry ({P}roc. {G}eom. {S}ympos.,
  {S}iegen, 1978)}, pages 164--185. Birkh\"auser, Basel, 1979.

\bibitem[Eck00]{eckhoff_survey}
J{\"u}rgen Eckhoff.
\newblock The partition conjecture.
\newblock {\em Discrete Math.}, 221(1-3):61--78, 2000.
\newblock Selected papers in honor of Ludwig Danzer.

\bibitem[JW81]{jamison_r3}
Robert~E. Jamison-Waldner.
\newblock Partition numbers for trees and ordered sets.
\newblock {\em Pacific J. Math.}, 96(1):115--140, 1981.
\newblock \url{http://projecteuclid.org/getRecord?id=euclid.pjm/1102734951}.

\bibitem[Rad21]{radon_orig}
Johann Radon.
\newblock Mengen konvexer {K}\"orper, die einen gemeinsamen {P}unkt enthalten.
\newblock {\em Math. Ann.}, 83(1-2):113--115, 1921.

\bibitem[SB83]{sierksma_boland}
Gerard Sierksma and Jan~Ch. Boland.
\newblock On {E}ckhoff's conjecture for {R}adon numbers; or how far the proof
  is still away.
\newblock {\em J. Geom.}, 20(2):116--121, 1983.

\bibitem[Tve66]{tverberg_orig}
H.~Tverberg.
\newblock A generalization of {R}adon's theorem.
\newblock {\em J. London Math. Soc.}, 41:123--128, 1966.

\bibitem[vdV93]{vandevel}
M.~L.~J. van~de Vel.
\newblock {\em Theory of convex structures}, volume~50 of {\em North-Holland
  Mathematical Library}.
\newblock North-Holland Publishing Co., Amsterdam, 1993.

\end{thebibliography}

\end{document}